\documentclass[a4paper,UKenglish]{lipics-v2018}

\usepackage{microtype}

\usepackage{amssymb}
\usepackage{color}

\newtheorem{claim}[theorem]{Claim}

\title{Almost all string graphs are intersection graphs of plane convex sets}
\titlerunning{Almost all string graphs are intersection graphs of plane convex sets}

\author{J\'{a}nos Pach}{Ecole Polytechnique F\'{e}d\'{e}rale de Lausanne and R\'{e}nyi Institute, Hungarian Academy of Sciences\\{[Lausanne, Switzerland and Budapest, Hungary]}}{pach@cims.nyu.edu}{}{[Supported by Swiss National Science Foundation Grants 200021-165977 and 200020-162884.]}
\author{Bruce Reed}{School of Computer Science, McGill University, Laboratoire I3S CNRS, and Professor Visitante Especial, IMPA\\{[Montreal, Canada and Rio de Janeiro, Brazil]}}{breed@cs.mcgill.ca}{}{}
\author{Yelena Yuditsky}{School of Computer Science, McGill University\\{[Montreal, Canada]}}{yuditskyL@gmail.com}{}{}

\authorrunning{J. Pach, B. Reed, Y. Yuditsky}
\Copyright{J\'{a}nos Pach, Bruce Reed and Yelena Yuditsky}

\category{}

\EventEditors{Bettina Speckmann and Csaba D. T{\'o}th}
\EventNoEds{2}
\EventLongTitle{34th International Symposium on Computational Geometry (SoCG 2018)}
\EventShortTitle{SoCG 2018}
\EventAcronym{SoCG}
\EventYear{2018}
\EventDate{June 11--14, 2018}
\EventLocation{Budapest, Hungary}
\EventLogo{socg-logo} 
\SeriesVolume{99}
\ArticleNo{68} %

\nolinenumbers
\begin{document}

\maketitle

\begin{abstract}
A {\em string graph} is the intersection graph of a family of continuous arcs in the plane. The intersection graph of a family of plane convex sets is a string graph, but not all string graphs can be obtained in this way. We prove the following structure theorem conjectured by Janson and Uzzell: The vertex set of {\em almost all} string graphs on $n$ vertices can be partitioned into {\em five} cliques such that some pair of them is not connected by any edge ($n\rightarrow\infty$). We also show that every graph with the above property is an intersection graph of plane convex sets. As a corollary, we obtain that {\em almost all} string graphs on $n$ vertices are intersection graphs of plane convex sets.
\end{abstract}

\section{Overview}

The {\it intersection graph} of a collection $C$ of sets is a graphs whose vertex set is $C$ and in which two sets in $C$ are connected by an edge if and only if they have nonempty intersection. A {\it curve} is a subset of the plane which is homeomorphic to the interval $[0,1]$. The intersection graph of a finite collection of curves (``strings'') is called a {\it string graph}.
\smallskip

Ever since Benzer \cite{Be59} introduced the notion in 1959, to explore the topology of genetic structures, string graphs have been intensively studied both for practical applications and theoretical interest. In 1966, studying electrical networks realizable by printed circuits, Sinden \cite{Si66} considered the same constructs at Bell Labs. He proved that not every graph is a string graph, and raised the question whether the recognition of string graphs is decidable. The affirmative answer was given by Schaefer and \v Stefankovi\v c \cite{ScSt04} 38 years later. The difficulty of the problem is illustrated by an elegant construction of Kratochv\'il and Matou\v sek \cite{KrMa91}, according to which there exists a string graph on $n$ vertices such that no matter how we realize it by curves, there are two curves that intersect at least $2^{cn}$ times, for some $c>0$. On the other hand, it was proved in \cite{ScSt04} that every string graph on $n$ vertices and $m$ edges can be realized by polygonal curves, any pair of which intersect at most $2^{c' m}$ times, for some other constant $c'$. The problem of recognizing string graphs is NP-complete~\cite{Kr91,ScSeSt03}.

\smallskip

In spite of the fact that there is a wealth of results for various special classes of string graphs, understanding the structure of general string graphs has remained an elusive task.
The aim of this paper is to show that {\em almost all} string graphs have a very simple structure. That is, the proportion of string graphs that possess this structure tends to $1$ as $n$ tends to infinity.

Given any graph property $\textsc{P}$ and any $n\in \mathbb{N}$, we denote by $\textsc{P}_n$ the set of all graphs with property $\textsc{P}$ on the (labeled) vertex set $V_n=\{1,\ldots,n\}$. 
In particular, ${\textsc{String}}_n$ is the collection of all string graphs with the vertex set $V_n$.
We say that an $n$-element set is partitioned into parts of {\em almost equal size} if the sizes of any two parts differ by at most $n^{1-\epsilon}$ for some $\epsilon>0$, provided that $n$ is sufficiently large.

\begin{theorem}\label{main}
As $n\rightarrow\infty$, the vertex set of almost every string graph $G\in {\textsc{String}}_n$ can be partitioned into $4$ parts of almost equal size such that $3$ of them induce a clique in $G$ and the $4$th one splits into two cliques with no edge running between them.
\end{theorem}

\begin{theorem}\label{main2}
Every graph $G$ whose vertex set can be partitioned into $4$ parts such that $3$ of them induce a clique in $G$ and the $4$th one splits into two cliques with no edge running between them, is a string graph.
\end{theorem}

Theorem \ref{main} settles a conjecture of Janson and Uzzell from \cite{JaU17}, where a related weaker result was proved in terms of graphons.
\smallskip

We also prove that a typical string graph can be realized using relatively simple strings.
\smallskip

Let ${\textsc{Conv}}_n$ denote the set of all intersection graphs of families of $n$ labeled convex sets $\{C_1,\ldots,C_n\}$ in the plane. For every pair $\{C_i, C_j\}$, select a point in $C_i\cap C_j$, provided that such a point exists. Replace each convex set $C_i$ by the polygonal curve obtained by connecting all points selected from $C_i$ by segments, in the order of increasing $x$-coordinate. Observe that any two such curves belonging to different $C_i$s intersect at most $2n$ times. The intersection graph of these curves (strings)  is the same as the intersection graph of the original convex sets, showing that ${\textsc{Conv}}_n\subseteq {\textsc{String}}_n$. Taking into account the construction of Kratochv\'il and Matou\v sek \cite{KrMa91}  mentioned above, it easily follows that the sets ${\textsc{Conv}}_n$ and ${\textsc{String}}_n$ are not the same, provided that $n$ is sufficiently large.

\begin{theorem}\label{distinction}
There exist string graphs that cannot be obtained as intersection graphs of convex sets in the plane.
\end{theorem}

We call a graph $G$ {\em canonical} if its vertex set can be partitioned into $4$ parts such that $3$ of them induce a clique in $G$ and the $4$th one splits into two cliques with no edge running between them. The set of canonical graphs on $n$ vertices is denoted by ${\textsc{Canon}}_n$. Theorem \ref{main2} states ${\textsc{Canon}}_n\subset {\textsc{String}}_n$. In fact, this is an immediate corollary of ${\textsc{Conv}}_n\subset{\textsc{String}}_n$ 
and the relation ${\textsc{Canon}}_n\subset{\textsc{Conv}}_n$, formulated as

\begin{theorem}\label{canonical}
The vertices of every canonical graph $G$ can be represented by convex sets in the plane such that their intersection graph is $G$.
\end{theorem}

The converse is not true. Every planar graph can be represented as the intersection graph of convex sets in the plane (Koebe ~\cite{Ko36}). Since no planar graph contains a clique of size exceeding four, for $n>20$ no planar graph
with $n$ vertices is canonical.

\smallskip

Combining Theorems~\ref{main} and~\ref{canonical}, we obtain the following.

\begin{corollary}
Almost all string graphs on $n$ labeled vertices are intersection graphs of convex sets in the plane.
\end{corollary}

See Figure \ref{fig:classes} for a sketch of the containment relation of the families of graphs discussed above. 
\begin{figure}
    \centering
                \includegraphics[width=0.45\textwidth]{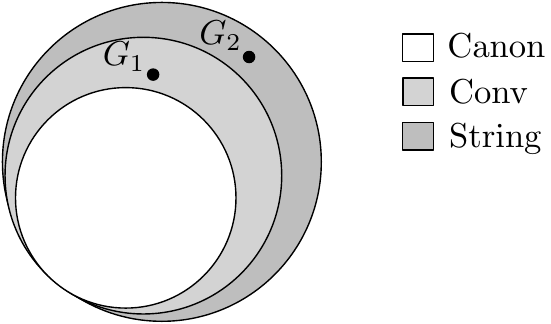}
                \caption{The graph $G_1$ is the any planar graph with more than $20$ vertices. The graph $G_2$ is the graph from the construction of Kratochv\'il and Matou\v sek \cite{KrMa91}.\label{fig:classes}}
\end{figure}

\medskip

The rest of this paper is organized as follows. In Section~\ref{speed}, we recall the necessary tools from extremal graph theory, and adapt a partitioning technique of Alon, Balogh, Bollob\'as, and Morris~\cite{AlBBM11} to analyze string graphs; see Theorem~\ref{abbm.cor}. Since the modifications are not entirely straightforward, we include a sketch of the proof of Theorem~\ref{abbm.cor} in the appendix. In Section~\ref{section3}, we collect some simple facts about string graphs and intersection graphs of plane convex sets, and combine them to prove Theorem~\ref{canonical}. In Section~\ref{firstproof}, we strengthen Theorem~\ref{abbm.cor} in two different ways and, hence, prove Theorem~\ref{main} modulo a small number of exceptional vertices. We wrap up the proof of Theorem~\ref{main} in Section~\ref{mainproof}.

\section{The structure of typical graphs in an  hereditary family}\label{speed}

A {\em graph property} ${\textsc P}$ is called {\em hereditary} if every induced subgraph of a graph $G$ with property ${\textsc P}$ has property ${\textsc P}$, too. With no danger of confusion, we use the same notation ${\textsc P}$ to denote a (hereditary) graph property and the family of all graphs that satisfy this property. Clearly, the properties that a graph $G$ is a string graph ($G\in{\textsc{String}}$) or that $G$ is an intersection graph of plane convex sets ($G\in{\textsc{Conv}}$) are hereditary. The same is true for the properties that $G$ contains no subgraph, resp., no induced subgraph isomorphic to a fixed graph $H$.

It is a classic topic in extremal graph theory to investigate the typical structure of graphs in a specific hereditary family. This involves proving that almost all graphs in the family have a certain structural decomposition. This research is inextricably linked to the study of the growth rate of the function $|{\textsc P}_n|$, also known as the {\it speed} of ${\textsc P}$, in two ways. Firstly, structural decompositions may give us bounds on the growth rate. Secondly, lower bounds on the growth rate help us to prove that the size of the exceptional family of graphs which fail to have a specific structural decomposition is negligible.
In particular, we will both use a preliminary bound on the speed in proving our structural result about string graphs, and apply our theorem to improve the best known current bounds on the speed of the string graphs.
\medskip

In a pioneering paper, Erd\H os, Kleitman, and Rothschild~\cite{ErKR76} approximately determined for every $t$ the speed of the property that the graph contains no clique of size $t$. Erd\H os, Frankl, and R\"odl~\cite{ErFR86} generalized this result as follows. Let $H$ be a fixed graph with chromatic number $\chi(H)$. Then every graph of $n$ vertices that does not contain $H$ as a (not necessarily induced) subgraph can be made $(\chi(H)-1)$-partite by the deletion of $o(n^2)$ edges. This implies that the speed of the property that the graph contains no subgraph isomorphic to $H$ is
\begin{equation}\label{perfect2}
2^{\left(1-\frac1{\chi(H)-1}+o(1)\right){n\choose 2}}.
\end{equation}

Pr\"omel and Steger~\cite{PrS92a, PrS92b, PrS93} established an analogous theorem for graphs containing no {\em induced subgraph} isomorphic to $H$. Throughout this paper, these graphs will be called {\em $H$-free}. To state their result, Pr\"omel and Steger introduced the following key notion.

\begin{definition}\label{coloringnumber}
A graph $G$ is {\em $(r,s)$-colorable} for some $0\le s\le r$ if there is a $r$-coloring of the vertex set $V(G)$, in which the first $s$ color classes are {cliques} and the remaining $r-s$ color classes are {independent sets}. The {\em coloring number} $\chi_c(\textsc{P})$ of a hereditary graph property $\textsc{P}$ is the largest integer $r$ for which there is an $s$ such that all $(r, s)$-colorable graphs have property $\textsc{P}$. Consequently, for any $0\le s\le \chi_c(\textsc{P})+1$, there exists a $(\chi_c(\textsc{P})+1,s)$-colorable graph that does not have property $\textsc{P}$.
\end{definition}

The work of Pr\"omel and Steger was completed by Alekseev~\cite{Al93} and by Bollob\'as and Thomason~\cite{BoT95, BoT97}, who proved that the speed of any hereditary graph property $\textsc{P}$ satisfies
\begin{equation}\label{alexeev}
|\textsc{P}_n|=2^{\left(1-\frac{1}{\chi_c({\rm P})}+o(1)\right){n\choose 2}}.
\end{equation}

The lower bound follows from the observation that for $\chi_c({\textsc P})=r$, there exists $s\le r$ such that all $(r,s)$-colorable graphs have property $\textsc{P}$. In particular, $\textsc{P}_n$ contains all graphs whose vertex sets can be partitioned into $s$ cliques and $r-s$ independent sets, and the number of such graphs is equal to the right-hand side of (\ref{alexeev}).

As for string graphs, Pach and T\'oth~\cite{PaT06} proved that
\begin{equation}\label{r=4}
\chi_c(\textsc{String})=4.
\end{equation}
Hence, (\ref{alexeev}) immediately implies
\begin{equation}\label{pachtoth}
          |\textsc{String}_n|=2^{\left(\frac34+o(1)\right){n\choose 2}}.
\end{equation}

If we want to tighten the above estimates, another idea of Pr\"omel and Steger~\cite{PrS91} is instructive. They noticed that the vertex set of almost every $C_4$-free graph can be partitioned into a clique and an independent set, and no matter how we choose the edges between these two parts, we always obtain a $C_4$-free graph. Therefore, the speed of $C_4$-freeness is at most $(1+o(1))2^n2^{\frac{1}{2}{n \choose 2}}$, which is much better than the general bound $2^{\left(\frac12+o(1)\right){n\choose 2}}$ that follows from (\ref{alexeev}). Almost all $C_5$-free graphs permit similar {\em ``certifying partitions''}. It is an interesting open problem to decide which hereditary families permit such partitions and what can be said about the inner structure of the subgraphs induced by the parts.
This line of research was continued by Balogh, Bollob\'as, and Simonovits~\cite{BaBS04, BaBS09, BaBS11}. The strongest result in this direction was proved by Alon, Balogh, Bollob\'as, and Morris~\cite{AlBBM11}, who proved that for almost every graph with a hereditary property $\textsc{P}$, one can delete a small fraction of the vertices in such a way that the rest can be partitioned into $\chi_c(\textsc{P})$ parts with a very simple inner structure. This allowed them to replace the bound (\ref{alexeev}) by a better one:
$$|\textsc{P}_n|=2^{\left(1-\frac{1}{\chi_c(\rm P)}\right){n\choose 2}+O(n^{2-\epsilon})}.$$
This will be the starting point of our analysis of string graphs.  As we  shall see, in the case of string graphs, our results allow us to replace the $2^{O(n^{2-\epsilon})}$ in this bound by $2^{\frac{9n}{4}+o(n)}$. See \cite{BB11, KKOT15, RY17, ReedScott}, for related results.
\smallskip

We need some notation. Following Alon {\em et al.}, for any integer $k>0$, define $U(k)$ as a bipartite graph with vertex classes $\{1,\ldots,k\}$ and $\{I : I\subset\{1,...,k\}\}$, where a vertex $i$ in the first class is connected to a vertex $I$ in the second if and only if $i\in I$. We think of $U(k)$ as a ``universal'' bipartite graph on $k+2^k$ vertices, because for every subset of the first class there is a vertex in the second class whose neighborhood is precisely this subset.

As usual, the {\em neighborhood} of a vertex $v$ of a graph $G$ is denoted by $N_G(v)$ or, if there is no danger of confusion, simply by $N(v)$. For any disjoint subsets $A, B\subset V(G)$, let $G[A]$ and $G[A,B]$ denote the subgraph of $G$ induced by $A$ and the {\em bipartite} subgraph of $G$ consisting of all edges of $G$ running between $A$ and $B$, respectively. The {\em symmetric difference} of two sets, $X$ and $Y$, is denoted by $X\bigtriangleup Y$.

\begin{definition}\label{U(k)-free}
Let $k$ be a positive integer. A graph $G$ is said to {\em contain $U(k)$} if there are two disjoint subsets $A,B\subset V(G)$ such that the bipartite subgraph $G[A,B]\subseteq G$ induced by them is isomorphic to $U(k)$. Otherwise, with a slight abuse of terminology, we say that $G$ is $U(k)$-free.
\end{definition}

By slightly modifying the proof of the main result (Theorem 1) in \cite{AlBBM11} and adapting it to string graphs, we obtain

\begin{theorem}\label{abbm.cor}
For any sufficiently large positive integer $k$ and for any $\delta>0$ which is sufficiently small in terms of $k$, there exist $\epsilon>0$ and a positive integer $b$ with the following properties.

The vertex set $V_n\; (|V_n|=n)$ of almost every string graph $G$ can be partitioned into eight sets, $S_1,...S_4,A_1,....,A_4$, and a set $B$ of at most $b$ vertices such that
\begin{enumerate}
\item[(a)] $G[S_i]$  is $U(k)$-free for every $i\; (1\le i\le 4)$;
\item[(b)] $|A_1 \cup A_2 ... \cup A_4 | \le n^{1-\epsilon}$; and
\item[(c)]for every $i\; (1\le i\le 4)$ and  $v\in S_i \cup A_i$ there is $a\in B$ such that $$|(N(v)\bigtriangleup N(a)) \cap (S_i \cup A_i)|\le \delta n.$$
\end{enumerate}
\end{theorem}

In other words, for the right choice of parameters, almost all string graphs have a partition into $4$ parts satisfying the following conditions. There is a set of sub-linear size in the number of vertices such that deleting its elements, the subgraphs induced by the parts are U(k)-free. Moreover, there is another set $B$ of at most constantly many vertices such that the neighborhood of every vertex with respect to the part it belongs to is similar to the neighbourhood of some vertex in $B$. 
In Appendix~\ref{sketch}, we sketch the proof of this result, indicating the places where we slightly deviate from the original argument in~\cite{AlBBM11}.

\section{String graphs vs. intersection graphs of convex sets--Proof of Theorem~\ref{canonical}}\label{section3}

Instead of proving Theorem~\ref{canonical}, we establish a somewhat more general result.

\begin{theorem}\label{planar}
Given a planar graph $H$ with labeled vertices $\{1,\ldots,k\}$ and positive integers $n_1,\ldots,n_k$, let $\textsc{H}(n_1,\ldots,n_k)$ denote the class of all graphs with $n_1+\ldots+n_k$ vertices  that can be obtained from $H$ by replacing every vertex $i\in V(H)$ with a clique of size $n_i$, and adding any number of further edges between pairs of cliques that correspond to pairs of vertices $i\neq j$ with $ij\in E(G)$.

Then every element of $\textsc{H}(n_1,\ldots,n_k)$ is the intersection graph of a family of plane convex sets.
\end{theorem}

\begin{proof}
Fix any graph  $G\in \textsc{H}(n_1,\ldots,n_k)$. The vertices of $H$ can be represented by closed disks $D_1,\ldots, D_k$ with disjoint interiors such that $D_i$ and $D_j$ are tangent to each other for some $i<j$ if and only if $ij\in E(H)$ (Koebe, \cite{Ko36}). In this case, let $t_{ij}=t_{ji}$ denote the point at which $D_i$ and $D_j$ touch each other. For any $i\; (1\le i\le k)$, let $o_i$ be the center of $D_i$. Assume without loss of generality that the radius of every disk $D_i$ is at least $1$.
\smallskip

$G$ has $n_1+\ldots+n_k$ vertices denoted by $v_{im}$, where $1\le i\le k$ and $1\le m\le n_i$. In what follows, we assign to each vertex $v_{im}\in V(G)$ a finite set of points $P_{im}$, and define $C_{im}$ to be the convex hull of $P_{im}$. For every $i, 1\le i\le k,$ we include $o_i$ in all sets $P_{im}$ with $1\le m\le n_i$, to make sure that for each $i$, all sets $C_{im}, 1\le m\le n_i$ have a point in common, therefore, the vertices that correspond to these sets induce a clique.
\smallskip

Let $\varepsilon<1$ be the {\em minimum} of all angles $\measuredangle t_{ij}o_it_{il}>0$ at which the arc between two consecutive touching points $t_{ij}$ and $t_{il}$ on the boundary of the same disc $D_i$ can be seen from its center, over all $i, 1\le i\le k$ and over all $j$ and $l$. Fix a small $\delta>0$ satisfying $\delta<\varepsilon^2/100$.
\smallskip

For every $i<j$ with $ij\in E(H)$, let $\gamma_{ij}$ be a circular arc of length $\delta$ on the boundary of $D_i$, centered at the point $t_{ij}\in D_i\cap D_j$. We select $2^{n_i}$ distinct points $p_{ij}(A)\in\gamma_{ij}$, each representing a different subset $A\subseteq\{1,\ldots, n_i\}$. A point $p_{ij}(A)$ will belong to the set $P_{im}$ if and only if $m\in A$. (Warning: Note that the roles of $i$ and $j$ are not interchangeable!)

If for some $i<j$ with $ij\in E(H)$, the intersection of the neighborhood of a vertex $v_{jM}\in V(G)$ for any $1\le M \le n_j$ with the set $\{v_{im} : 1\le m\le n_i\}$ is equal to $\{v_{im} : m\in A\}$, then we include the point $p_{ij}(A)$ in the set $P_{jM}$ assigned to $v_{jM}$, see Figure \ref{Figure2} for a sketch. Hence, for every $m\le n_i$ and $M\le n_j$, we have
$$v_{im}v_{jM}\in E(G)\;\;\;\Longleftrightarrow\;\;\;P_{im}\cap P_{jM}\neq\emptyset.$$
In other words, the intersection graph of the sets assigned to the vertices of $G$ is isomorphic to $G$.
\smallskip
\smallskip

\begin{figure}
    \centering
                \includegraphics[width=0.55\textwidth]{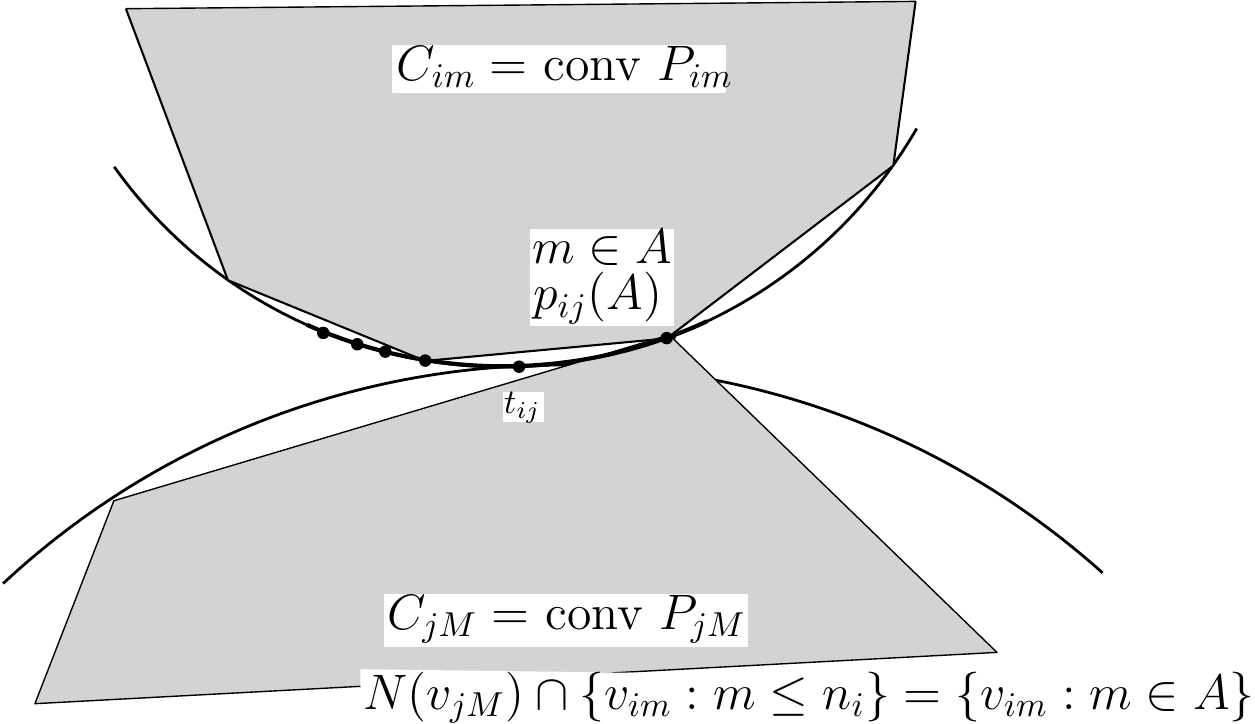}
                \caption{The point $p_{ij}(A)$ is included in $P_{jM}$. \label{Figure2}}
\end{figure}

It remains to verify that
$$v_{im}v_{jM}\in E(G)\;\;\;\Longleftrightarrow\;\;\;C_{im}\cap C_{jM}\neq\emptyset.$$
Suppose that the intersection graph of the set of convex polygonal regions
$$\{C_{im} : 1\le i\le k\; {\rm and}\; 1\le m\le n_i\}$$
differs from the intersection graph of
$$\{P_{im} : 1\le i\le k\; {\rm and}\; 1\le m\le n_i\}.$$

Assume first, for contradiction, that there exist $i, m, j, M$ with $i<j$ such that $D_i$ and $D_j$ are tangent to each other and $C_{jM}$ contains a point $p_{ij}(B)$ for which
\begin{equation}\label{badpoint}
B\not=N_{jM}\cap\{v_{im} : 1\le m\le n_i\}.
\end{equation}
Consider the unique point $p=p_{ij}(A)\in\gamma_{ij}$ that belongs to $P_{jM}$, that is, we have
$$A=N_{jM}\cap\{v_{im} : 1\le m\le n_i\}.$$
Draw a tangent line $\ell$ to the arc $\gamma_{ij}$ at point $p$. See Figure \ref{Figure1}. The polygon $C_{jM}$ has two sides meeting at $p$; denote the infinite rays emanating from $p$ and containing these sides by $r_1$ and $r_2$. These rays either pass through $o_j$ or intersect the boundary of $D_j$ in a small neighborhood of the point of tangency of $D_j$ with some other disk $D_{j'}$. Since $\delta$ was chosen to be much smaller than $\varepsilon$, we conclude that $r_1$ and $r_2$ lie entirely on the same side of $\ell$ where $o_j$, the center of $D_j$, is. On the other hand, all other points of $\gamma_{ij}$, including the point $p_{ij}(B)$ satisfying (\ref{badpoint}) lie on the opposite side of $\ell$, which is a contradiction.

\begin{figure}
    \centering
                \includegraphics[width=0.55\textwidth]{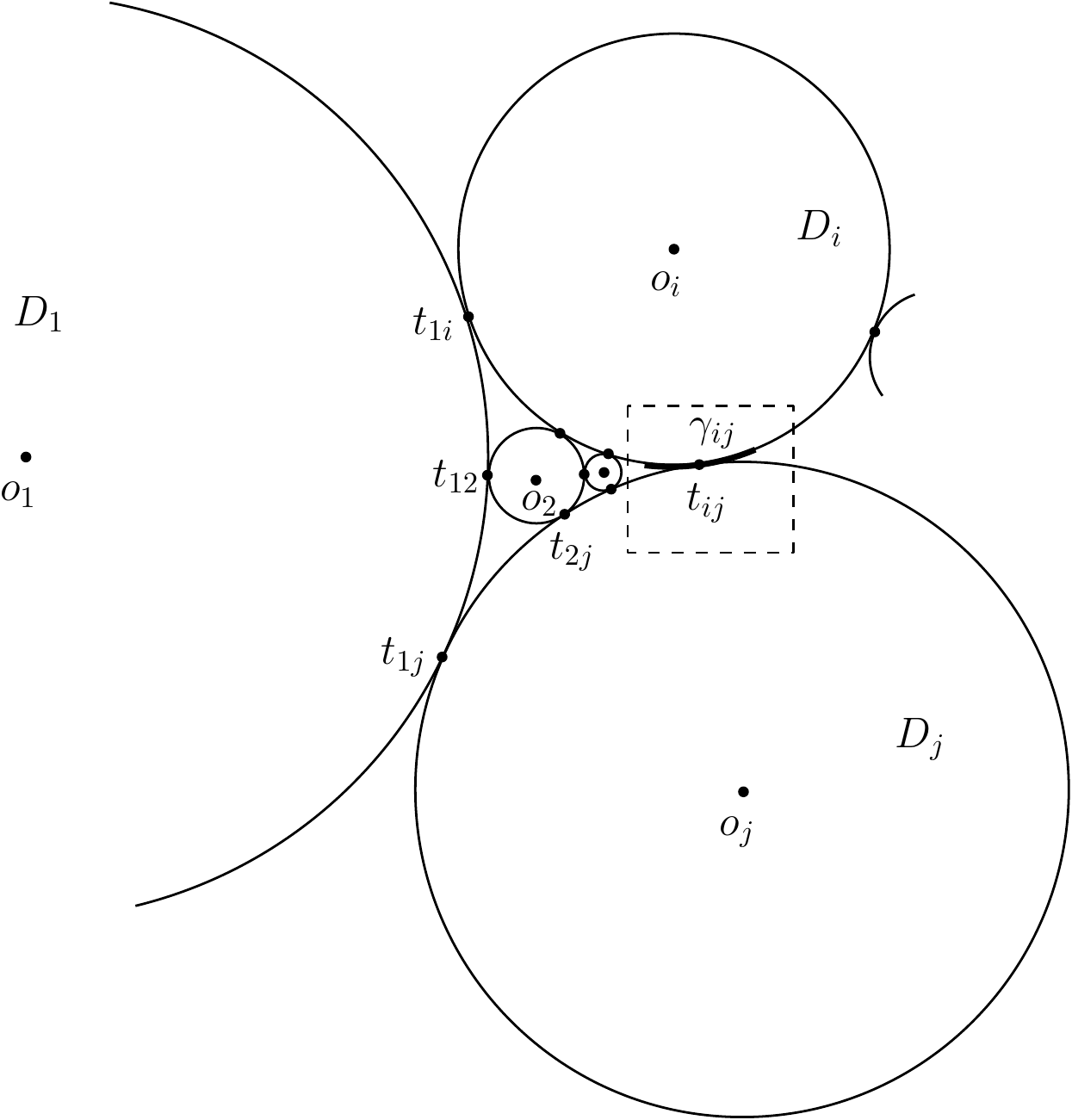}
                \caption{Tangent disks $D_i$ and $D_j$ touching at $t_{ij}$.\label{Figure1}}
\end{figure}

Essentially the same argument and a little trigonometric computation show that for every $j$ and $M$, the set $C_{jM}\setminus D_j$ is covered by the union of some small neighborhoods (of radius $<\varepsilon/10$) of the touching points $t_{ij}$ between $D_j$ and the other disks $D_i$. This, together with the assumption that the radius of every disk $D_i$ is at least $1$ (and, hence, is much larger than $\varepsilon$ and $\delta$) implies that $C_{jM}$ cannot intersect any polygon $C_{im}$ with $i\not= j$, for which $D_i$ and $D_j$ are not tangent to each other.
\end{proof}

\smallskip

Applying Theorem~\ref{planar} to the graph obtained from $K_5$ by deleting one of its edges, Theorem~\ref{canonical} follows.

%
%
%
%
%
%
%
%
%
%

\section{Strengthening Theorem~\ref{abbm.cor}}
\label{firstproof}

In this section, we strengthen Theorem~\ref{abbm.cor} in two different ways. To avoid confusion, in the formulation of our new theorem, we use $X_i$ in place  of $S_i$ and $Z_i$ in place of $A_i$. We will see that we can insist that the four parts of the partition have approximately the same size. Secondly, we can guarantee that $X_1$, $X_2$, and $X_3$ are cliques and $X_4$ induces the disjoint union of two cliques. More precisely, setting $Z=Z_1 \cup Z_2 ... \cup Z_4$, we prove the following result, which is similar in flavour to a result in \cite{ReedScott}.

\begin{theorem}
\label{Hstruc.thm}
For every sufficiently small $\delta$, there are  $\gamma>0, b>4+\frac{2}{\delta}$ with the following property. For almost every string graph $G$ on $V_n$, there is a partition of $V_n$ into $X_1,...,X_4, Z_1,...,Z_4$ such that for some set $B$ of at most $b$ vertices the following conditions are satisfied:
\begin{enumerate}
\item[(I)] $G[X_1]$, $G[X_2]$, and $G[X_3]$ are cliques and $G[X_4]$ induces the disjoint union of  two cliques.
\item[(II)] $| Z_1 \cup Z_2 \cup Z_3 \cup Z_4| \le n^{1-\gamma}$, 
\item[(III)] for every $i\; (1\le i\le 4)$ and every $v\in X_i \cup Z_i$, there exists $a\in B$ such that $$|(N(v)\bigtriangleup N(a)) \cap (X_i \cup Z_i)|\le \delta n,$$
\item[(IV)] for every $i\; (1\le i\le 4)$, we have $\bigl\lvert |Z_i \cup X_i|-\frac{n}{4}\bigr\rvert \le n^{1-\gamma}$.
\end{enumerate}
\end{theorem}

See Figure \ref{Figure4} for an illustration of Theorem \ref{Hstruc.thm}.

\begin{figure}
    \centering
                \includegraphics[width=0.8\textwidth]{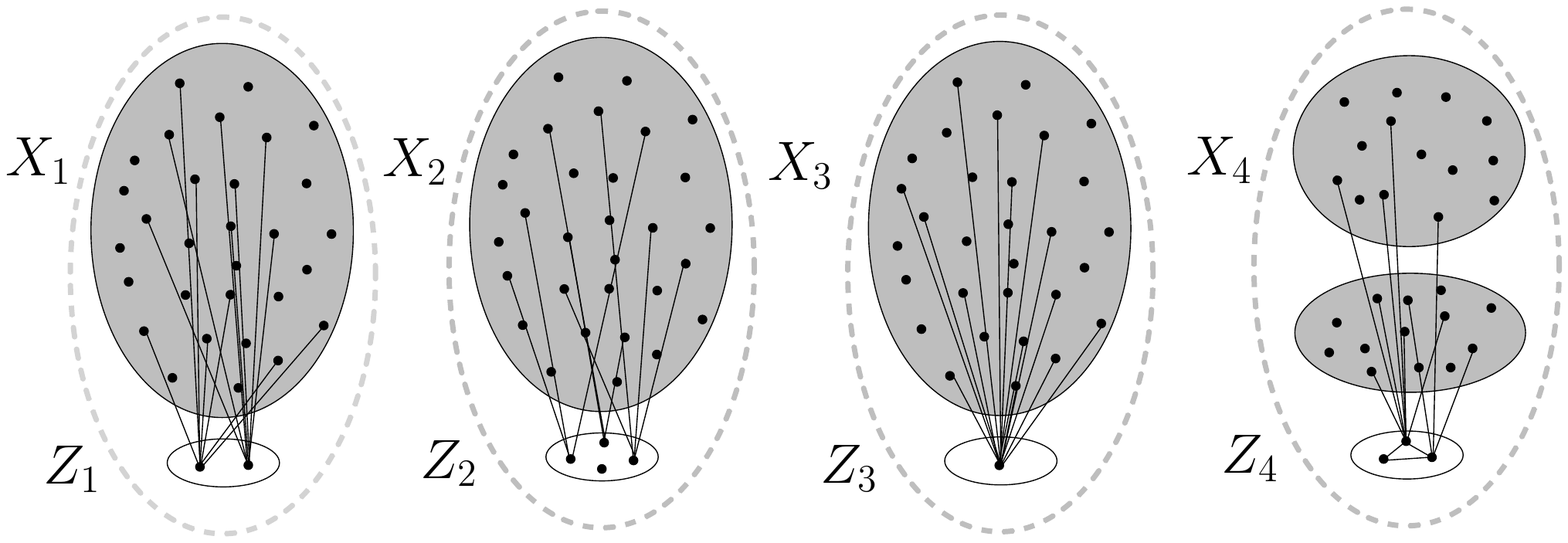}
                \caption{A sketch of a typical string graph as in Theorem \ref{Hstruc.thm}. The edges between the parts are not drawn. The sets shaded grey are cliques.\label{Figure4}}
\end{figure}

For the proof of Theorem \ref{Hstruc.thm} we need the following statement which is a slight generalization of Lemma~3.2 in~\cite{PaT06}, and it can be established in precisely the same way,
details are given in the appendix.

\begin{lemma}\label{nonstring}
Let $H$ be a graph on the vertex set $\{v_1,\ldots,v_5\}\cup\{v_{ij} : 1\le i\not= j\le 5\}$, where $v_{ij}=v_{ji}$ and every $v_{ij}$ is connected by an edge to $v_i$ and $v_j$. The graph $H$ may have some further edges connecting pairs of vertices $(v_{ij},v_{ik})$ with $j\not=k$.
Then $H$ is not a string graph.
\end{lemma}\label{non-string}

\begin{corollary}\label{thecor}
For each of the following types of partition, there exists a non-string graph whose vertex set
can be partitioned in the specified way:

(a) $2$ stable (that is, independent) sets each of size at most 10;

(b) $4$ cliques each of size at most five and a vertex;

(c) $3$ cliques  each of size at most five and a stable set of size $3$;

(d) $3$ cliques each of size at most five and a path with three vertices;

(e) $2$ cliques both of size at most five  and $2$ graphs that can be obtained as the disjoint union of a point and a clique of size at most $3$.
\end{corollary}

See Figure \ref{Figure3} for an illustration of Corollary \ref{thecor}.
\begin{figure}
    \centering
                \includegraphics[width=0.75\textwidth]{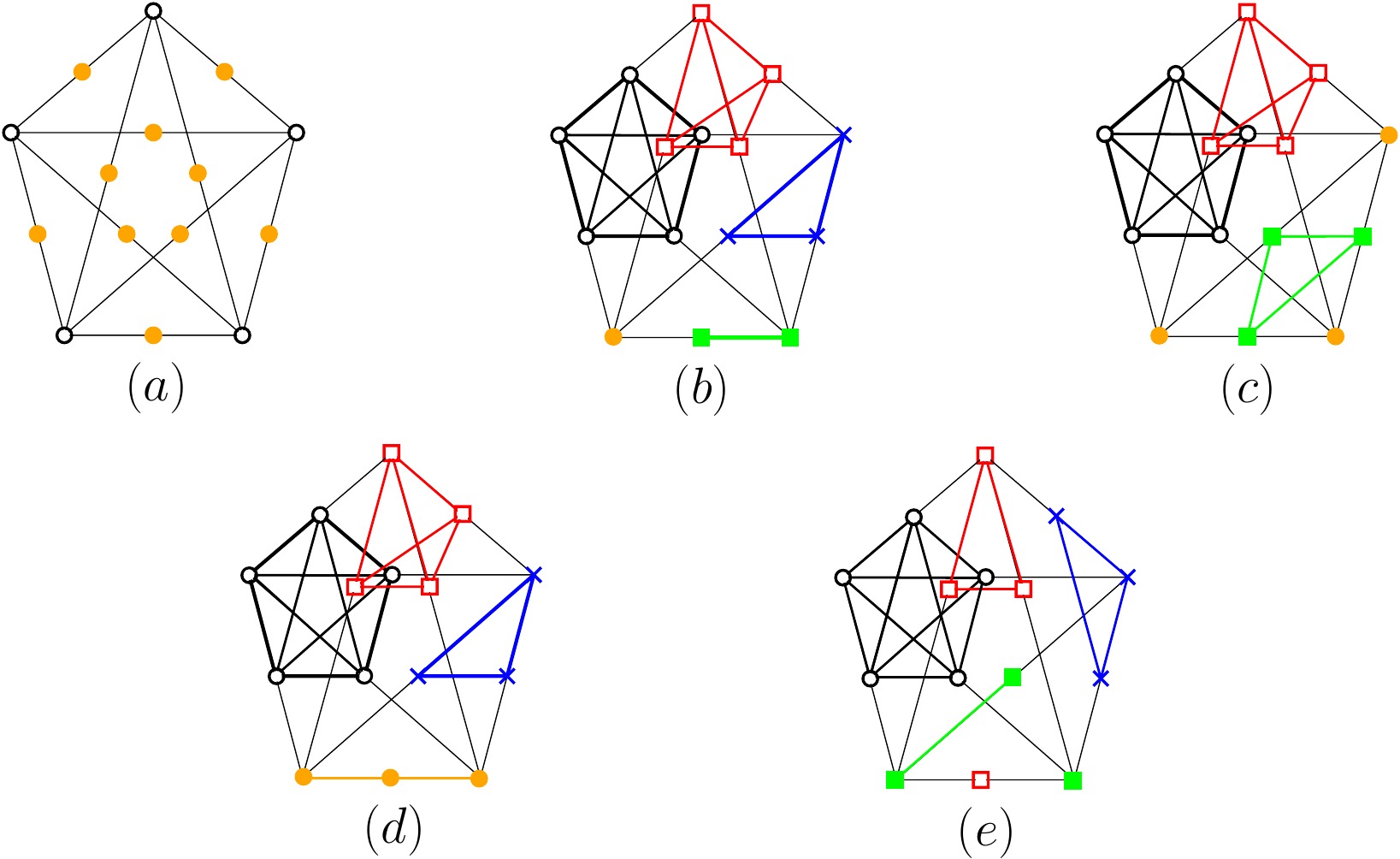}
                \caption{Possible partitions of a non-string graph.\label{Figure3}}
\end{figure}

\begin{proof} [Proof of Theorem \ref{Hstruc.thm}.]
We choose $k$ sufficiently large and then $\delta<\frac{1}{40}$ sufficiently small in terms of $k$.
We choose $\epsilon,b>0$ such that Theorem \ref{abbm.cor} holds for this choice of $k$ and $\delta$ and  so that $\epsilon$ is less than the $\rho$ of Lemma \ref{restrictbound.cor} for this choice of $k$. We set $\gamma= \frac{\epsilon}{10}$ and
consider $n$ large enough to satisfy certain implicit inequalities below. We  know that the subset ${\cal S}(k,\delta)_n$ of  ${\textsc{String}}_n$,
 consisting of those graphs for which there is a  set $B$ of at most $b$ vertices and a partition into  $S_i$ and $A_i$ satisfying (a),(b), and (c) set out in  Theorem \ref{abbm.cor},
contains almost every string graph.  We call such a partition, {\it certifying}.
We need to show that  almost every  graph in ${\cal S}(k,\delta)_n$ has a certifying partition for which we can repartition $S_i \cup A_i$ into $X_i \cup Z_i$ so that  (I),(II), and (IV)  all hold
(that (III) holds, is simply Theorem \ref{abbm.cor} (c) and $S_i \cup A_i=X_i \cup Z_i$).

We prove this  fact via a sequence of lemmas. In doing so, for a specific partition, we let $m=m(A_1 \cup S_1,A_2 \cup S_2,A_3 \cup S_3,A_4 \cup S_4)$  be the number of pairs of vertices not lying together in some $A_i \cup S_i$.
The first  lemma gives us a lower bound  on $|{\cal S}(k,\delta)_n|$, obtained by simply counting the number of graphs which permits a partition into four cliques all of size
within one of $\frac{n}{4}$. Its four line proof is given in the appendix.

\begin{lemma}
\label{countingclaim}
$|{\cal S}(k,\delta)_n| \ge  2^\frac{3{n \choose 2}}{4}$.
\end{lemma}

The second gives us an upper bound on the number of choices for $G[S_i]$ for graphs $G$ in  ${\cal S}(k,\delta)_n$
for which $S_1,S_2,S_3,S_4,A_1,A_2,A_3,A_4$ is a certifying partition. It is Corollary 8 in \cite{AlBBM11}.

\begin{lemma}
\label{restrictbound.cor}
For every k, there is a positive  $\rho$ such that for every sufficiently large $l$, the number of
$U(k)$-free  graphs with $l$ vertices  is less than $2^{l^{2-\rho}}$.
\end{lemma}

Next we prove:

\begin{lemma}
\label{deviationlema}
The number of graphs in  ${\cal S}(k,\delta)_n$ which have a certifying partition  such that for some $i$,  $||A_i \cup S_i|-\frac{n}{4}| > n^{1-\gamma}$
is $o(|{\cal S}(k,\delta)_n|)$.
\end{lemma}

\begin{proof}
The number of choices for a  partition of $V_n$ into $S_1,S_2,S_3,S_4,A_1,A_2,A_3,A_4$   is at most
$8^n$. If this partition demonstrates that   $S_i$ is $U(k)$-free and $n$ is large, Lemma \ref{restrictbound.cor}  tells us that there  are only $2^{n^{2-\epsilon}}$ choices for $G[S_i]$. The number of  choices for the edges out of each vertex of $A_i$ is $2^{n-1}$. So, since $|A_i|$ is  at most $n^{1-\epsilon}$,  we know there are at most $2^{n^{2-\epsilon}}$ choices for the edges out of $A_i$. It follows that there are
at most $2^{11(n^{2-\epsilon})}$ choices for our partition and the
graphs $G[S_1 \cup A_1],....,G[S_4 \cup A_4]$ over all $G$ in ${\cal S}(k,\delta)_n$  which can be certified using this partition.
Furthermore,the number of graphs  in ${\cal S}(k,\delta)_n$ permitting such a  certifying choice is  at most $2^m$. Since, $|{\cal S}(k,\delta)_n| \ge 2^\frac{3{n \choose 2}}{4}$,
it follows that  almost every graph  $G$ in ${\cal S}(k,\delta)_n$ has no  certifying partition for which $m<\frac{3{n \choose 2}}{4}-12(n^{2-\epsilon})$. The desired
result follows.
\end{proof}

Setting  $l=l_n=\lceil n^{1-\frac{\epsilon}{7}} \rceil$, we have the following.

\begin{lemma}
\label{firststruclema}
The number of graphs in  ${\cal S}(k,\delta)_n$ which have a certifying partition  for which  there are distinct $i$ and $j$ such that  both $S_i$ and $S_j$ contain
$l$ disjoint independent sets of size 10 is $o(|{\cal S}(k,\delta)_n|)$.
\end{lemma}

\begin{proof}
Consider  a choice of certifying partition and induced subgraphs $H_1,H_2,H_3,H_4$ where $V(H_i)=A_i \cup S_i$.
By Corollary \ref{thecor}(a),  for any  pair of independent sets of size $10$,  at least one of the $2^{100}$ choices of edges
between the sets yields  a bipartite non-string graph. Thus, the number of choices for edges between the
partitions  which  extend our choice to yield a graph in $String _n$ is at most   $2^m(1-\frac{1}{2^{100}})^{l^2}$.
Since $m<\frac{3 {n \choose 2}}4$ and $l^2= \omega(n^{2-\frac{\epsilon}{2}})$,
it follows that for almost every graph in ${\cal S}(k,\delta)_n$, almost every certifying partition  does not contain two distinct such $i$ and $j$.
\end{proof}

Ramsey theory tells us that if a graph $J$  does not contain $l$ disjoint   stable sets of size 10, it contains $|V(J)|-10(l-1)-2^{15}$ disjoint cliques of size 5. Combining applications of this fact to three of the $G[S_i]$, Corollary 11(c),  and an argument similar to that used in the proof of Lemma \ref{firststruclema} allows  us to prove the following lemma. Details can be found in the appendix.
\begin{lemma}
\label{secondstruclema}
The number of graphs $G$ in ${\cal S}(k,\delta)_n$ which have a certifying partition for which there is an $i=i(G)$ such that $S_i$ does not contain $l$ disjoint cliques of size 5 is $o(|{\cal S}(k,\delta)_n|)$
\end{lemma}

With this lemma in hand, we can mimic the argument used in its proof to obtain the following two lemmas.
In doing so, we apply Corollary 11 (c),(d), and (e).

\begin{lemma}
\label{thirdstruclema}
The number of graphs $G$  in  ${\cal S}(k,\delta)_n$ which have a certifying partition  for which  there is an $i=i(G)$ such that  $S_i$ contains $l$ disjoint
sets of size three  each inducing a stable set  or a path
is $o(|{\cal S}(k,\delta)_n|)$.
\end{lemma}

\begin{lemma}
\label{fourthstruclema}
The number of graphs $G$  in  ${\cal S}(k,\delta)_n$ which have a certifying partition  for which  there are  two distinct  $i$ such that  $S_i$ contains $l$ disjoint
sets of size four  each inducing the disjoint union of a vertex and a triangle  is $o(|{\cal S}(k,\delta)_n|)$.
\end{lemma}

Combining these lemmas, and possibly permuting indices, we see that almost every graph in ${\cal S}(k,\delta)_n$ has a certifying partition for which
for every $i \le 4$ we have $||Z_i \cup X_i|-\frac{n}{4}| \le n^{1-\gamma}$,  no $S_i$ contains more than $l$ sets inducing  a path of length three or  a stable set of size three,
and for every $k \le 3$, $S_k$  does not contain $l$ disjoint sets inducing the disjoint union of a vertex and a triangle. For each such graph, we
consider such a partition. For all  $i<4$, we let $Z_i$ be the union of $A_i$  and a maximum family of disjoint sets in $X_i$ each inducing a path of length 3, a stable set of  size three, or the disjoint union of a triangle and a vertex. We let $Z_4$ be the union of $A_4$  and a maximum family of disjoint sets in $X_4$ each inducing a path of length three or a stable set of size three. We set $X_i=S_i -Z_i$.
\end{proof}

\section{Completing the proof of Theorem~\ref{main}}
\label{mainproof}

In this section, we prove our main result. By a   {\it great} partition  of $G$ we mean a partition of its vertex set into
$X_1,X_2, X_3,X_4$ such that for $i\le 3$, $X_i$ is a clique and $X_4$ is the disjoint union of two cliques. We call 
a graph {\it great} if it has a great partition and {\it mediocre} otherwise.  
Theorem \ref{main} simply states  that almost every string graph  $G$ on $V_n$ is great.

Thus, we are trying to show that almost every string graph has a partition  into sets $X_1,X_2,X_3,X_4,Z_1,Z_2,Z_3,Z_4$ 
satisfying  Theorem \ref{Hstruc.thm} (I) with the sets $Z_i$ empty.
 We  choose $\delta$ so small that Theorem~\ref{Hstruc.thm} holds and $\delta$ also satisfies certain inequalities implicitly given below.  We apply Theorem~\ref{Hstruc.thm} and obtain that for some positive  $\gamma$ and $b$, for  almost every  graph in ${\textsc{String}}_n$  there is  a partition  of $V_n$ into $X_1, ..., X_4,Z_1,..,Z_4$ satisfying
(I), (II), (III), and (IV). Note that if we reduce $\gamma$ the theorem remains true. We insist that $\gamma$ is at most
$\frac{1}{64000000}$. We call such partitions {\it good}.
We need  to show that the number of mediocre  string graphs
on $V_n$ with a good partition  is of smaller order than the number of great graphs on $V_n$.

The following result tells us that the number of great graphs on $V_n$ is of the same order as the number of
great  partitions of  graphs on $V_n$.

\begin{claim}
\label{ourclaim}
The ratio between the number of great partitions of graphs on $V_n$ and the number of graphs which permit such partitions is $6+o(1)$.
\end{claim}

So, it is sufficient to show that the number  of mediocre  string graphs with a good partition
on $V_n$  is of smaller order than the number of graphs with a great  partition  on $V_n$.
In doing so, we consider each partition separately. 
For every partition  ${\cal Y}=(Y_1,Y_2,Y_3,Y_4)$   of $V_n$ we say that a  good partition  
satisfying (I)-(IV) with $Y_i=X_i \cup Z_i$ for every $i$ is ${\cal Y}$-good.
We prove:

\begin{claim}
 \label{ourotherclaim}
For every partition  ${\cal Y}=(Y_1,Y_2,Y_3,Y_4)$   of $V_n$, the number of  graphs which permit
a  great partition with $X_i=Y_i$  for every $i$  is of larger order then the size of the set ${\cal F}={\cal F}_{\cal Y}$ of  mediocre string graphs
which permit a ${\cal Y}$-good partition. 
\end{claim}

To complete the proof of Theorem~\ref{main} we need to  show that our two claims hold.

Before doing so, we deviate momentarily and discuss the speed of the string graphs.
Combining Theorem \ref{main} and Claim \ref{ourclaim}, we see that the  ratio of  the size of $|\textsc{String}_n|$
over the number of  ordered great partitions of graphs on $V_n$ is $\frac{1}{6}+o(1)$, so we need only count the latter.
There are $2^{2n}$  ordered partitions of $V_n$ into $Y_1,...,Y_4$, and there are $2^{m+|Y_4|}$ graphs for which this is
a great partition, where, as before, $m=m(Y_1,Y_2,Y_3,Y_4)$ is the number of pairs of vertices not lying together in some $Y_i$. This latter term is at most $2^{\frac{3 {n \choose 2}}{4}+\frac{n}{4}}$, which gives us the claimed upper bound on the
speed of string graphs. Furthermore, a simple calculation of the $2^{2n}$ ordered 4-partitions of $V_n$ shows that there is an $\Omega  (\frac{1}{n^\frac{3}{2}})$
proportion where no two parts differ in size by more than one. This gives us the claimed lower bound.

We now  prove our two claims.
In proving both, we exploit the fact that if a string graph has a great partition and we fix the subgraph induced by the parts of the partition, then any choice we make for the edges between the sets $X_i$ will yield another string graph permitting the same  great partition.

This fact implies that the edge arrangements between the partition elements of a graph permitting a particular great partition are chosen uniformly at random and, hence, are unlikely to lead to a graph permitting some other great partition.
This allows us to prove Claim \ref{ourclaim}, which we do  in the appendix.

\begin{proof}[\bf Proof  ~of  Claim \ref{ourotherclaim}:]

Let $m$ be the number of
pairs of vertices not contained in a  partition element and note that  there are exactly   $(2^{|Y_4|-1})$  choices for $G[Y_4]$ for a graph for
which ${\cal Y}$ is a great partition, and hence $2^m(2^{|Y_4|-1})$   graphs for which ${\cal Y}$ is a great partition.  

Our approach is to show that while there may be more choices for the $G[Y_i]$ for mediocre graphs for which ${\cal Y}$ is a good partition,
for each such choice we have many fewer than $2^m$ choices for mediocre string graphs extending these subgraphs.

We note that by the definition of good, we need only consider partitions such that each $Y_i$ has size $\frac{n}{4}+o(n)$.

Let $G\in {\cal F}$ and let $P(G)$ be the projection of $G$ on the sets $(Y_1,Y_2,Y_3,Y_4)$, that is, the disjoint union of the sets $G[Y_1],G[Y_2],G[Y_3]$, and $G[Y_4]$. 

Now, (I) of Theorem \ref{Hstruc.thm} bounds the number of choices for $G[Y_i]$ by 1 if $i<3$ and $2^{|Y_4|}$ if $i=4$.
Furthermore, (III) bounds the number of edges out of $Z_i$ in terms of its size and (II) bounds its size. Putting this all together we obtain the following lemma. Its proof can be found in the appendix.

\begin{lemma}\label{projcount}
Let $(Y_1,Y_2,Y_3,Y_4)$ be a partition of $V_n$, the number of possible projections on $(Y_1,Y_2,Y_3,Y_4)$ of graphs in ${\cal F}$ is $o(2^{nb+1+\sqrt{\delta}n|Z|})=o(2^{|Y_4|-1}\cdot 2^{\sqrt{\delta}n^{2-\gamma}}).$
\end{lemma}

For a mediocre graph $G$ in ${\cal F}$, we call a set $D$ \textit{versatile} if for each  $i\in [4]$ with $Y_i\cap D=\emptyset$, there is  clique $C_i$  in $Y_i$  such that for all subsets  $D'$ of $D$ there are  $\frac{n}{\log n}$ vertices of $C_i$ which are adjacent to all elements of $D'$ and to none of $D\setminus D'$.

\begin{lemma}
\label{firststruclemma}
The number of mediocre string graphs in ${\cal F}$ such that for some $i$ there is a  versatile subset $T_i$
 of  3 vertices of $Y_i$ inducing a path  or a stable set of size three,is $o(2^m)$.
\end{lemma}

\begin{proof}
To begin, we count the number of mediocre graphs which extend a  given projection on $(Y_1,Y_2,Y_3,Y_4)$ where $T_i$ induces such a graph. We first  expose the edges from $Y_i$ to determine if  $T_i$ is versatile and then count the number of choices for the remaining edges between the partition elements. If $T_i$ is versatile we choose cliques  $C_k$ which show this is the case. 

By  Corollary \ref{thecor} (c) or (d), there is a non-string graph
$J$ whose vertex set can be partitioned into 3 cliques of size at most five, and a graph $J_i$  isomorphic to 
the subgraph  of the projection induced  by $T_i$. We label these three cliques as $J_k$ for $k \in \{1,2,3,4\}-\{i\}$
and let $f$ be an isomorphism from $J_i$ to $T_i$.  
For each vertex $v \in V(J_k)$, let $N(v)=f(N_J(v)\cap V(J_i))$ and $Z_v$ be those vertices of $C_k$ whose 
neighbourhhod on $T_i$ is $N(v)$. 
  Now, since  $|Z_v| \ge \frac{n}{ \log n}$ for all $v$ in each $V(J_k)$,   for each $k \neq i$, we can choose $n'= \lceil \frac{n}{10 \log n} \rceil $ cliques of size at most five $C^k_1,...,C^k_{n'}$
such that there is bijection $h_{k,l}$ from $J_k$ to $C^k_l$ with $h_{k,l}(v) \in Z_v$ for every $v \in J_k$.

If we choose our cliques in this way then for any set of three cliques $\{C^k_{i(k)} |k \neq i\}$ there is a choice of edges between the cliques which would make
the union of these three cliques with $T_i$ induce $J$. Thus, there is one choice of edges between the cliques which cannot be used in any extension of $H$
to a string graph. Mimicking an earlier argument, this implies that the number of choices for edges between the partition elements which extend $H$ to a string graph
is at most $2^{m-\frac{n^2}{\log^3n}}$. 
By the bound in Lemma \ref{projcount} on the number of possible projections, the desired result follows.
\end{proof}

 Using Corollary \ref{thecor}
(e) in places of (c) \& (d), we can ( and do in the appendix)  prove an analogous result for sets of size 8  intersecting two partition elements. To state it we need a definition. A graph $J$  is {\it extendible} if  there is some non-string graph  whose vertex set  can be partitioned into two cliques of size five and a set inducing $J$.  

\begin{lemma}
\label{secondstruclemma}
The number of mediocre string graphs in ${\cal F}$ such that for some distinct $i$ and $k$ there are subsets $T_i$ of $Y_i$ and $T_k$ of $Y_k$, both of size four, whose union is  both versatile and  induces
an extendible  graph  is $o(2^m)$.
\end{lemma}

For every mediocre string graph $G$  in ${\cal F}$, we choose a maximum family ${\cal W}={\cal W}_G$ of disjoint sets each of which is either (a) contained in some $Y_i$ and induces one of a stable set of size three or a path of length three,
or (b)  contains exactly four vertices from each of  two distinct partition elements  and  is extendible.
For every such choice we count the number of elements of ${\cal F}$ whose projection yields the 
given choice of ${\cal W}$.
 
Now, by the definition  of a good partition, each $Y_k$ contains a clique $C_k$ containing half the vertices of 
$X_k$ and hence at least $\frac{n}{10}$ vertices.  Lemmas \ref{firststruclemma} and \ref{secondstruclemma}  imply that  we can restrict our attention to graphs  for which for any subset $T$ in ${\cal W}$,
 there is a subset $N$ of $T$ and a $j $ with $Y_j$ disjoint from $T$  such that there are fewer than 
$\frac{n}{log n}$ vertices of $C_k$ which are adjacent to all of $N$ and none of $T-N$. 
This implies that  the number of choices for the edges 
from $T$ to other partition elements is $o(2^{\frac{3n|T|}{4}-\frac{n}{10000}})$.

Every element of ${\cal W}$ must intersect $Z$, so that $|{\cal W}| \le  |Z|$.  Set $W^*=\cup_{W \in {\cal W}} W$, and  let $Y'_i=Y_i-W^*$. Note  that for every $i$, $Y'_i$ has more than $\frac{n}{5}$ vertices and $G[Y'_i]$ is the disjoint union of two cliques.
Given a choice of ${\cal W}$, the number of choices for projections on $V_n\setminus W^*$ is less than $2^n$.
Mimicking the proof of Lemma \ref{projcount}, the number of choices for the vertices of $W^*$, and the edges of $G$ from the vertices in $W^*$ which remain within the partition elements of ${\cal Y}$  is $O(2^{bn+ \sqrt{\delta} |W^*|n})$. 
Combining this with the result of the  last  paragraph yields:  

\begin{lemma}
There is a constant $C$ such that the number of mediocre string graphs in ${\cal F}$  for which  $|{\cal W}|>C$ is $o(2^{m+|Y_4|})$.
\end{lemma}

So, we can restrict our attention to mediocre  graphs which have a partition  for which $|{\cal W}| \le C$. Similar tradeoffs allow us to handle them. Full details are found in the Appendix. 
 \end{proof}

\section{Acknowledgement}

This research was carried out while all three authors were visiting IMPA in Rio de Janeiro.
They would like to thank the institute for its generous support.

\appendix
\section{The Appendix}

\subsection{Sketch of the proof of Theorem~\ref{abbm.cor}}\label{sketch}

\begin{proof}
We only need to prove this result for $\delta$ sufficiently small as it then follows for all $\delta$.
We will set $\delta$ to be $3\alpha$  for some $\alpha$ which is required to be sufficently small. So, we can and do  replace $\delta$ by $3 \alpha$
in what follows.
We essentially follow the \cite{AlBBM11} proof of their Theorem 1  given in Section 7 of their paper.
We note that  our statement  differs from their
statement in the following ways (i) for us  the hereditary family ${\cal P}$ is the family of string graphs hence, as Pach and Toth proved $\chi_c(P)=4$,  (ii)
we allow $k$ to be any large enough integer rather than one fixed large integer, (ii) we allow $\alpha$ to be arbitrarilly small as long as it is small enough
in terms of $k$( and ${\cal P}$), (iii) $\epsilon$ is chosen as a function of $\alpha$ and $k$,  (iv) there is an integer $b$  which is  chosen as a function of
$\alpha$ and $k$ such that there is a choice $B$ of at most $b$ vertices and a partition of $A$ into $A_1,A_2,A_3,A_4$ for which our property (c) holds, and (v) the sentence beginning {\it Moreover} is deleted.    We will not reproduce the entire proof. We simply set out  the very minor modifications
these changes require.

  We want  to use the  strengthening  of their Lemma 23 obtained by replacing {\it and  $\alpha=\alpha(k,{\cal P})>0$ such that} in its statement with  {\it such that for any
$\alpha$ sufficiently small in terms of $k$ and ${\cal P}$}, and
(iii) replacing {\it $|B|$ with}  in  the definition of  $U(P_n,\alpha,k)$ just before the statement of Lemma 23  by
{\it  $|B|$ with  $|B| >  c(\alpha,{\cal P})$ for the $c$ of Lemma 18 or}.
 Their proof of the lemma actually proves this strengthening,  provided that  (a) in the first paragraph we set out that $c$ is the
$c(\alpha,{\cal P})$ of Lemma 18, (b) replace  $n^{1-2\alpha}$ by $c$ in the definition of $U_n$
given on its fourth line, and (iii) delete {\it if $c=c(\alpha,{\cal P})$ is sufficeintly large}.

Now while  following their (three paragraph) proof of their Theorem 1, we again replace $\alpha=\alpha(k, {\cal P})$ by $\alpha>0$ sufficiently small in terms of $k$ and ${\cal P}$,
 and  insist $\epsilon, \delta$ and $\gamma$ are sufficiently small in terms of both these parameters. Furthermore,
we define $c$ to be the $c(\alpha, {\cal P})$ of Lemma 18.
We also add {\it and $|B| \le c$}  at the end of the second paragraph before {\it for almost every}.

Then we consider the adjustment $S^\prime_1,...,S^\prime_r$ and exceptional set $A$ they obtain  and set  $A_i=S^\prime_i \cap A, S_i=S_i^\prime-A$. Now, as in their proof, consider a maximal $2\alpha$ bad set
$B$.  By our strengthened version of Lemma 23 the size of $B$ is at most  $c$ We set $b$ to be this $c$.
Now, (a) is their Theorem 1(b), (b) is their Theorem 1 (a) where $\epsilon$ is $\frac{\alpha}{2}$, and (c) follows immediately from the fact that
$S^\prime_1,...,S^\prime_4$ is an $\alpha$-adjustment and the definition of $\gamma$-adjustment.

\end{proof}

\subsection{The Proof of Lemma \ref{nonstring}}

\begin{proof}
Suppose for contradiction that $H$ has a string representation. Continuously contract each of string curve representing $v_{i}\; (1\le i\le 5)$ to a point $p_i$, without changing the intersection pattern of the curves. For every pair $i\neq j,$ consider  some   non-self intersecting arc  of the curve representing $v_{ij}$ with endpoints  $p_i$ and $p_j$. These arcs define a drawing of $K_5$, in which no two independent edges intersect. However, $K_5$ is not a planar graph, hence, by a well known theorem of Hanani and Tutte
\cite{Ch34}, \cite{Tu70}, no such drawing exists.
\end{proof}

\subsection{The Proof of Lemma \ref{countingclaim}}

\begin{proof}
For any partition of $V_n$ into  four sets $S_1,S_2,S_3,S_4$, each of size between $\frac{n-3}{4}$ and  $\frac{n+3}{4}$, there are at least $2^\frac{3{n \choose 2}}{4}$ string graphs on $n$ vertices in which the partition elements form cliques. We note these graphs are in ${\cal S}(k,\delta)_n$ with  the $A_i$ empty and $B$ containing one vertex from each clique.
So $|{\cal S}(k,\delta)_n| \ge 2^\frac{3{n \choose 2}}{4}$.
\end{proof}

\subsection{The Proof of Lemma \ref{secondstruclema}}

\begin{proof}

By Lemmas \ref{deviationlema} and \ref{firststruclema}, it is enough to consider graphs in ${\cal S}(k,\delta)_n$ with respect to which every $S_i$ contains more than $\frac{n}{5}$ vertices and there are no two distinct $k\neq t$ such that $S_k$ and $S_t$ contain $l$ disjoint stable sets of size $10$.


By Ramsey theorem, every set of $2^{15}$ vertices in any $S_j$ contains either a clique of size $5$ or stable sets of size $10$.
By our assumption $S_i$ does not contain $l$ disjoint cliques of size $5$, therefore for large enough $n$ it must contain $l$ disjoint stable sets $Z^i_1,...,Z^i_l$ of size $10$. Therefore, for all $j\neq i$, $S_j$ does not contain $l$ disjoint stable sets of size $10$, and, hence, it contains a set $Z^j_1,...,Z^j_l$ of $l$ cliques of size $5$.

By Corollary \ref{thecor}, for  the union of  any one of the independent sets  in $S_i$ and a clique of size five from each of the other $S_j$, there is  choice of  edges  between the partition element which  extends these stable sets and cliques to   a non-string graph induced by the 25 vertices. Now,  for some prime $p$ between $\frac{l}{2}$ and $l$,   we consider $p^2$ such unions   given by, for each $1 \le r,s\le p$: $Z^1_r,Z^2_{r+s},Z^3_{r+2s},Z^4_{r+3s}$ (where addition is modulo $p$).  We see that the number of choices for edges between the partition elements which gives a string graph, given the choices for the edges within is at most $2^{\frac{3{n \choose 2}}{4}} (1-\frac{1}{2^{225}})^{p^2}$.
The desired result follows.
\end{proof}

\subsection{The Proof of Claim \ref{ourclaim}}

\begin{proof}
To prove our claim, we focus on  graph--great partition pairs $(G,(X_1,X_2,X_3,X_4))$, that is, where the partition $(X_1,X_2,X_3,X_4)$  is a great partition of $G$ with the following property:
\vskip0.12cm

(P*)
\begin{enumerate}
\item[(a)] any two vertices of $G$ in the same partition element $X_i$ which forms a clique, have at least
$\frac{13n}{32}$ common neighbours;
\item[(b)] two vertices in different  partition elements  have fewer than
$\frac{13n}{32 }$ common neighbours;
\item[(c)]  for every partition element $X_i$ and every vertex $v$ not in $X_i$, $v$ forms a path of length three with two vertices of $X_i$; and
\item[(d)] $X_4$ does not induce a clique.
\end{enumerate}

Clearly, every great graph has at least six great partitions obtained by permuting the indices of the partition
elements.
We show now that  (i) every graph on $V_n$  has  at most six great  partitions   satisfying   (P*), and  (ii) almost every graph--great partition pair  on $V_n$  satisfies (P*). These two statements  prove our claim.

To prove (i), we assume that  $\{X_1,X_2,X_3,X_4\}$ and $\{X'_1,X'_2,X'_3,X'_4\}$
are two great partitions of a graph $G$, both of which satisfy property (P*). Clearly,  (a) and (b) tell us that for $i\le 3$, $X_i$ is contained in some $X'_j$.
Now, (c) tells us that each such $X_i$ is, in fact, nonempty and equal to some  $X_j$.
Hence, the set of partition elements is the same. Therefore, by (d), $X'_4=X_4$ and (i) follows.

It remains to show (ii). For any (ordered) partition ${\cal X}=X_1,X_2,X_3,X_4$  of $V_n$,
let $C_1=C_1({\cal X})$ be all choices of edges  within the partition elements which result in this partition being great. As before, let $m=m({\cal X})$ be the number of pairs of vertices not lying in a
partition element.

There are $|C_1|2^m$ graphs for which this partition is great, as we can pair any choice from $C_1$ with any choice
of edges between the partition  elements.  Furthermore, $C_1$ can be chosen by specifying a partition of $X_4$ into two disjoint cliques.  Thus, there is at least one and  at  most $2^{n-1}$ choices for $C_1$. Since  there are fewer than $4^n$ choices for ${\cal X}$ and $m$ decreases as the partition becomes more unbalanced,
for almost every graph-great partition pair we have that  for
each $i$, $|X_i|=\frac{n}{4}+o(n^\frac{2}{3})$ and we need only show that each fixed partition
having this property satisfies (P*) for almost  every graph for which it yields a great partition.

Since  we know that $|X_4|=\frac{n}{4}+o(n^\frac{2}{3})$, and almost every graph  on $n'$ vertices which is the disjoint union of two cliques is not a clique,  for almost every choice of the edges in $C_1$, for any choice of the edges  between the partition elements, we obtain a graph  satisfying (d).  We restrict our attention to the subset of $C_1$ for which (d) holds.

Now, we can choose a great graph extending this choice of $C_1$ uniformly
at random, by adding each edge joining vertices in different partition elements independently with probability $\frac{1}{2}$.

We observe that given a set of three vertices $u,v,w$ which is not contained in any $X_i$, the probability that $w$ is a common neighbour of $u$ and $v$ is at most $\frac{1}{2}$ if $w$ lies in the same partition element as
one of $u$ or $v$ and exactly $\frac{1}{4}$ otherwise. Taking into account the restriction on the choices in $C_1$ we consider, we obtain that the expected number of common neighbours of two vertices
is at most $\frac{1}{4} \cdot \frac{2n}{4}+ \frac{1}{2}\cdot\frac{2n}{4}+o(n)=\frac{3n}{8}+o(n)$ if they are in different partition elements, and at least
$\frac{n}{4}+ \frac{1}{4}\cdot\frac{3n}{4}+o(n)=\frac{7n}{16}+o(n)$ if they are in the same partition element which induces a clique. So,  for every choice in the subset of $C_1$ to which we have restricted ourselves,
 $n \choose 2$ applications of the Chernoff Bound, one for each pair of vertices, show that the proportion of
great graphs extending this partition on which  one of (P*)(a) or (P*)(b) fails is o(1).

In the same vein, consider an $X_i$ and a vertex $v$ outside of $X_i$. We partition $X_i$ into $\frac{|X_i|}{2}$ disjoint pairs of vertices.
For each pair, there is a choice of edges between this pair and $v$ for which these three vertices induce a path.
Thus, when we randomly construct a great graph extending $C_1$,
the probability that none of these sets of three  vertices induces a path is less than $(\frac{1}{4})^\frac{n}{6}$.
Since there are fewer than $n$ choices for $v$ and only $4$ choices for $X_i$, it follows that (c) holds for almost all great graphs extending $C_1$.
This proves (ii) and our claim.
\end{proof}

\begin{proof}
Any such mediocre string graph $G$, yields a corresponding projection $P(G)$, where $T_i$ induces a path of length three, or a stable set of size three.  We count the number of all mediocre graphs which extend a projection on $(Y_1,Y_2,Y_3,Y_4)$ with such a set $T_i$.
In doing so, we exploit the fact that for $k \neq i$, there is a clique $C_k$ which contains a third of every $Z_{k,N}$ (as we could choose $C$ to be at least half of the vertices in $X_k$ if we had specified $X_k$).

By  Corollary \ref{thecor}, there is a non-string graph
$J$ whose vertex set can be partitioned into 3 cliques of size at most five, and a graph $J_i$  isomorphic to $G^*[T_i]=H[T_i]$. We label these three cliques as $J_k$ for $k \in \{1,2,3,4\}-\{i\}$
and let $f$ be an isomorphism between $J_i$ and $T_i$.  
For each vertex $v \in V(J_k)$, let $N(v)=f(N_J(v)\cap V(J_i))$.
We let $Z_{v}^k=Z_{k,N(v)} \cap C_k$.  Now, since each $Z_v$ has at least $\frac{n}{2 \log n}$ elements, for each $k \neq i$, we can choose $n'= \lceil \frac{n}{10 \log n} \rceil $ cliques of size at most five $C^k_1,...,C^k_{n'}$
such that there is bijection $h_{k,l}$ from $J_k$ to $C^k_l$ with $h_{k,l}(v) \in Z_v$ for every $v \in J_k$.

If we choose our cliques in this way then for any set of three cliques $\{C^k_{i(k)} |k \neq i\}$ there is a choice of edges between the cliques which would make
the union of these three cliques with $T_i$ induce $J$. Thus, there is one choice of edges between the cliques which cannot be used in any extension of $H$
to a string graph. Mimicking an earlier argument, this implies that the number of choices for edges between the partition elements which extend $H$ to a string graph
is at most $2^{m-\frac{n^2}{\log^3n}}$. 
By the bound in Lemma \ref{projcount} on the number of possible projections, the desired result follows.
\end{proof}

\section{Completing The Proof of Claim \ref{ourotherclaim}}

In this section, we complete the proof of Claim \ref{ourotherclaim}. 
We begin with the promised proof of Lemma \ref{projcount} which we restate for the reader's convenience. 

\begin{lemma}\label{aprojcount}
Let $(Y_1,Y_2,Y_3,Y_4)$ be a partition of $V_n$, the number of possible projections on $(Y_1,Y_2,Y_3,Y_4)$ of graphs in ${\cal F}$ is $o(2^{nb+1+\sqrt{\delta}n|Z|})=o(2^{|Y_4|-1}\cdot 2^{\sqrt{\delta}n^{2-\gamma}}).$
\end{lemma}

\begin{proof}
We can specify a projection by specifying the vertices of
$Z=Z_1 \cup ..Z_4$ and the edges out of them, along with the partition of $X_4$ into two cliques.
We can choose the edges out of the vertices in  $Z_i$ by first choosing the neighbourhoods of the $b$ vertices of $B$
and then assigning each vertex of $Z_i$ to one of these $b$ vertices and specifying the  at most $\delta n$ vertices
in the symmetric difference of the neighbourhoods of these two vertices.
So, there are at most $(2^{|X_4|-1})2^{nb}b^{{|Z|}}{n \choose |Z|}{n \choose \delta n}^{|Z|}$ choices for $P(G)$ over all $G$ in ${\cal F}$. We note that  for $\delta$ sufficiently small this is $o(2^{nb+1+\sqrt{\delta}n|Z|})$. This is   $o(2^{\sqrt{\delta}n^{2-\gamma}})$ because $|Z|\le n^{1-\gamma}$ by part (III) of Theorem \ref{Hstruc.thm}.
\end{proof}

We next give the promised proof of Lemma \ref{secondstruclemma} which we restate for the reader's convenience. 

\begin{lemma}
\label{asecondstruclemma}
The number of mediocre string graphs in ${\cal F}$ such that for some distinct $i$ and $k$ there are subsets $T_i$ of $Y_i$ and $T_k$ of $Y_k$, both of size four, whose union is  both versatile and  induces
an extendible  graph  is $o(2^m)$.
\end{lemma}

\begin{proof}
Any such mediocre string graph $G$, yields a corresponding projection $P(G)$, where $T_i \cup T_k$ induces an
extendible graph.   To begin, we count the number of mediocre graphs which extend a given projection on $(Y_1,Y_2,Y_3,Y_4)$ where $T_i \cup T_k$ induces such a graph. We first  expose the edges from $Y_i \cup Y_k$ to determine if  $T_i \cup T_k$ is versatile and then count the number of choices for the remaining edges between the partition elements. If  $t_i \cup T_k$ is versatile we choose cliques $C_l$ which show this is the case. 

By  Corollary \ref{thecor} (c) or (d), there is a non-string graph
$J$ whose vertex set can be partitioned into 2 cliques of size at most five, and a subgraph $J_i$   isomorphic to $G^*[T_i \cup T_k]$. We label these two cliques as $J_l$ for $l \in \{1,2,3,4\}-\{i,k\}$
and let $f$ be an isomorphism from $J_i$ to $T_i \cup T_k$.  
For each  $l \in \{1,2,3,4\} -\{i,k\}$ and vertex $v \in V(J_l)$, let $N(v)=f(N_J(v)\cap V(J_i))$ and $Z_v$ be those vertices of $C_l$ whose 
neighbourhhod on $T_i \cup T_k$ is $N(v)$. 
  Now, since  $|Z_v| \ge \frac{n}{ \log n}$ for all $v$ in each $V(J_l)$,   for each $l \not\in \{ i,k \}$, we can choose $n'= \lceil \frac{n}{10 \log n} \rceil $ cliques of size at most five $C^l_1,...,C^l_{n'}$
such that there is bijection $h_{l,r}$ from $J_l$ to $C^l_r$ with $h_{l,r}(v) \in Z_v$ for every $v \in J_l$.

If we choose our cliques in this way then for any pair iof cliques $\{C^l_{r(l)} |l \not\in \{i,k\}\}$ there is a choice of edges between the cliques which would make
the union of these two cliques with $T_i \cup T_k$ induce $J$. Thus, there is one choice of edges between the cliques which cannot be used in any extension of $H$
to a string graph. Mimicking an earlier argument, this implies that the number of choices for edges between the partition elements which extend $H$ to a string graph
is at most $2^{m-\frac{n^2}{\log^3n}}$. 
By the bound in Lemma \ref{projcount} on the number of possible projections, the desired result follows.
\end{proof}

We recall that
for every mediocre string graph $G$  in ${\cal F}$, we chose a maximum family ${\cal W}={\cal W}_G$ of disjoint sets each of which is either (a) contained in some $Y_i$ and induces one of a stable set of size three or a path of length three,
or (b) contains exactly four vertices from each of  the two partition elements it intersects and is extendible. 
We set $W^*= \cup_{W \in {\cal W}}W$ and $Y'_i=Y_i-W^*$.
We proved that there was an absolute constant $C$ such that the number of mediocre string graphs in ${\cal F}$  for which $|{\cal W}|>C$ is $o(2^{m+|Y_4|})$. 

Thus, it remains to show

\begin{lemma}
\label{athirdstruclemma}
For any $C$, the  number of mediocre string graphs in ${\cal F}$  for which $|{\cal W}| \le C$,
and no element of ${\cal W}$ is versatile,  is $o(2^{m+|Y_4|})$.
\end{lemma}

\begin{proof}

We note that  ${\cal W}$ is nonempty as we are considering mediocre graphs.
Further, by the maximality of ${\cal W}$, each $Y'_i$ is the disjoint union of two cliques. 
We note further that the number of projections for which ${\cal W}$ has at most $C$ 
elements is at most $4^n{n \choose 8C}2^{Cn}=2^{O(n)}$.

We bound first those graphs for which there are distinct $i$ and $j$ such that both  $Y'_i$ and $Y'_j$ 
contain two components larger than $n^\frac{2}{3}$. 
In this case, for $k \in \{i,j\}$ we can find a set ${\cal Z}_k$ of $\frac{n^\frac{2}{3}}{2}$ disjoint sets each inducing the disjoint union of a triangle and a vertex.
Now, Corollary \ref{thecor} (e) and our choice of $W$ tells us that for each pair of sets, one from each $Z_k$, there is a choice of edges between the
two sets which cannot occur in a mediocre graph of the type we are counting. Thus, the total number of such mediocre graphs is at most $2^{m+O(n)}(1-\frac{1}{2^{16}})^\frac{n^\frac{4}{3}}{2}=o(2^m)$.
This implies the desired result, in this case.

For each vertex $v$ in $W^*$, the {\em rank of $v$ with respect to a partition element $Y_i$} is
$\max  \{ \min (|N(v) \cap K|, |K-N(v)|) |~K~ is~ a~ component~of~Y_i-W^* ~with~at ~least \frac{2n}{\log n}~vertices \}$.
We use $rank(v)$ to
denote the minimum  of these ranks over the partition elements. We say $v$ is {\it extreme} on $Y_i$ if its rank
with respect to $Y_i$ is less than $\frac{n}{log n}$.

We consider next the case that  our mediocre string graph contains  a vertex $v$ in $W^*$ which is not extreme on any partition element.   In order to count such graphs we first expose the projection $P(G)$ on our partition and the
choice of $W^*$. We then expose the edges out of $W^*$ to determine which of its elements are extreme 
on the various partition elements. We then bound the choices for the other edges between the partition elements 
given our current choice. We note that we make $2^{O(n)}$ choices initially. 
If  for some such choice, some $v \in W^*$ is not extreme to any partition element, we can and do  choose $p=\frac{n}{5log n}$ $P_3s$ all containing $v$, but otherwise disjoint and contained in $Y_1$. Let $T^1_1,...,T^1_p$ be this set of $P_3$s.

By part (d) of Corollary \ref{thecor}, there is a non-string graph $J$ which can be partitioned into $(J_1,J_2,J_3,J_4)$, where $J_1=P_3$ and $J_2,J_3,J_4$ are cliques of size at most $5$. 
For each $1 \le j \le p$, let $f_j$ be an isomorphism from $J_1$ to $T^1_j$ such that the pre-image of $v$ is the same under all $f_j$. For each $j \ge 2$,  let $n_j$ be the
number of vertices of $J_j$ that are adjacent to the preimage of $v$. Because $v$ is not extreme on any $Y_j$, for each such $j$ we can choose a set
$T^j_1,..,T^j_p$ of disjoint cliques of $Y_j$, each  of size five and containing $n_j$ neighbours of $v$.

If we choose our cliques in this way, then for any   set $T^1_j-v$, $j\in [p]$, and for any choice of a clique from each of our sets, there is a choice of edges between these sets which would make
the union of these four sets and $v$  induce $J$. Thus, there is a choice of edges between the cliques which cannot be used in any extension of to a string graph. Mimicking an earlier argument, this implies that the number of choices for edges between the partition elements which extend $H$ to a string graph
is at most $2^{m-\frac{n^2}{log^3~n}}$.  The desired result
follows.

It remains to consider the case when every vertex of $W^*$ is extreme on some partition element. 
Here we fix  a choice for $W^*$, and  a choice of 
the set of partition elements to which each element of $W^*$ is extreme. We note that this is 
$O({n \choose |W^*|})$ choices. We then count the number of extensions of   these choices to a 
mediocre string graph for which ${\cal Y}$ is a good partition.  

We let $W^*_2$ be those vertices of $W^*$ which are extreme on at least
two partition elements and let $W^*_1$ be those vertices of $W^*$ which are extreme on exactly one partition element.  We consider a new partition ${\cal Y}^*=Y^*_1,...,Y^*_4$
obtained by moving each element of  $W^*$ to a $Y_i$ to which it has rank equal to $rank(v)$. 
Since $Y^*_i-W^*=Y_i-W^*$,
because we are in this case we know that there are at most $2^{max\{|Y_i|, 1 \le i \le 4\}}{ n \choose n^\frac{2}{3}}^3=2^{|Y_4|+o(1)}$ 
choices for the edges of  such a medicore string graph  which lie within the $Y^*_i-W^*$. 
We note further that because  each $Y_i$ has size near $\frac{n}{4}$ and we move only a constant number of vertices, the difference between $m$ and the number $m'$ of pairs of vertices lying in different  elements of this new partition is $O(n^{1-\epsilon})$.

We note that for each vertex of $W^*_2$ there are at most $2^{n/2+o(n)}$ choices for the edges out of it.
We let $v$ be a vertex of $W^*_1$ minimizing $rank(v)$ over all vertices of rank greater than zero.
Providing such a $v$ exists, mimicking the argument for a $v$ which is not extreme to any partition element, we can show that the number of choices for edges between the  $Y^*_i$
is $2^{m-\Omega(\frac{n\cdot rank(v)}{\log n})}$. On the other hand, treating $C$ as an absolute constant,  the number of choices for the edges from
the vertices of $W^*_1$ within the partition elements is
$O\left(\left({n \choose rank(v)}^2\cdot 2^\frac{2n}{log n}\right)^{|W^*_1|}\right) =2^{O(rank(v) log~n+\frac{n}{\log n})}$. Combining this with the results of the last paragraph we see that we are done
unless $W^*_2$ is empty and every vertex of $W^*_1$ has rank which is $o(\log n)$.

But, now there are $2^{o(n)}$  choices for the edges from the vertices of $W^*$   within the $Y^*_i$.
Furthermore, since for $i \le 3$, $Y_i-Z_i$ is a clique, we see that there are fewer than
$n \choose 2n^{1-\epsilon}$ choices for the edges of $Y^*_i-W^*_i$. Thus, $Y^*_4-W^*_4$
must have two components each of size at least $\frac{n}{10}$ or we are done.  Hence, we can find a
family ${\cal Z}$ of $\frac{n}{30}$ disjoint sets within it each inducing the disjoint union of a clique and a triangle.
Furthermore, there are
$O(2^{|Y_4|+o(n)})$ choices for the projections of the graphs we are counting. So, Lemmas \ref{firststruclemma}  and \ref{secondstruclemma} imply that
we are done unless  for every $i$,  $Y^*_i$ induces the disjoint union of two  cliques, and there  are no eight vertices, intersecting
two $Y^*_i$ each in the disjoint union of a clique and a triangle which induce an extendible graph. Since the graphs we count are mediocre,
for some $i\le 3$, $Y^*_i$ is not a clique and so contains a set $T$ of four vertices which induce the disjoint union of a
vertex and a triangle. Hence, by  Corollary \ref{thecor} (e), for every $Z$ in  ${\cal Z}$ there is a choice of an edge set
between $T$ and $Z$ which cannot occur in the graphs we are counting.
We are done.
\end{proof}

\end{document}